\documentclass[12pt]{extarticle}
\usepackage[margin=1.5in]{geometry}
\usepackage{mathtools,amsthm}
\usepackage{enumitem}
\usepackage{titlesec}
\usepackage{hyperref}
\usepackage{comment}
\usepackage{tikz-cd}
\usepackage{amsfonts}
\titleformat{\section}{\normalfont\Large\bfseries}{\thesection}{1em}{}
\titleformat{\subsection}{\normalfont\large\bfseries}{\thesubsection}{1em}{}
\theoremstyle{plain}
\newtheorem{thm}{Theorem}[section]

\newtheorem*{question*}{Question}
\newtheorem*{questions*}{Questions}
\newtheorem{lemma}[thm]{Lemma}
\newtheorem*{thm*}{Theorem}
\newtheorem*{prop*}{Proposition}
\newtheorem{prop}[thm]{Proposition}

\newtheorem{cor}[thm]{Corollary}
\theoremstyle{definition}
\newtheorem{defi}[thm]{Definition}

\theoremstyle{remark}
\newtheorem{rmk}[thm]{Remark}

\newcommand{\Oo}{\mathcal{O}}

\DeclareMathOperator{\Pic}{Pic}

\DeclareMathOperator{\Spec}{Spec}

\newcommand{\NS}{\operatorname{NS}}

\newcommand{\Z}{\mathbb{Z}}
\newcommand{\Q}{\mathbb{Q}}

\newcommand{\Pp}{\mathbb{P}}

\newcommand{\p}{\mathfrak{p}}

\newcommand{\Ff}{\mathbf{F}}

\usepackage{authblk}

\newcommand{\Char}{\mathrm{char}}

\numberwithin{equation}{section}
\begin{document}
\title{On the Chow group of elliptic surfaces over number fields}
\author{Domenico Valloni}
\affil{\'{E}cole polytechnique f\'{e}d\'{e}rale de Lausanne}
\maketitle
\begin{abstract}
Let $X$ be a smooth projective surface over a number field $K$. Assume that $X$ has an elliptic fibration over $\Pp^1_K$ with at least one singular fibre and a section. Let $\mathcal{X}/U$ be a smooth projective model of $X$ over some open subset $U \subset \Spec(\Oo_K)$. We show that $\ker\bigl(\mathrm{CH}^2(\mathcal{X}) \rightarrow \mathrm{CH}^2(X)\bigr)$ is a finitely generated group.
\end{abstract}
\section{Introduction}
Let $X$ be a smooth projective variety over a number field $K$. A conjecture originally attributed to Swinnerton-Dyer (see \cite{zbMATH04061376}[Section 5] and also \cite{MR1086888} and \cite{MR1159204}) says that the Chow groups $\mathrm{CH}^i(X)$ are finitely generated. At present, this conjecture is widely open, and we do not even know whether the $n$-torsion subgroup $\mathrm{CH}^2(X)[n]$ is finitely generated when $X$ is a surface, although by Rojtman's theorem \cite{MR577137} this is true over the algebraic closure of $K$. The strongest evidence concerning the finiteness of $\mathrm{CH}^2(X)_{\mathrm{tors}}$ is due to \cite{CT-R}[Theorem C], where the authors show that $\mathrm{CH}^2(X)_{\mathrm{tors}}$ is finitely generated whenever $H^2(X, \Oo_X) = 0$ (see also \cite{MR1300892}). We refer the reader to \cite{MR1744949} and \cite{OTS}, which contain surveys regarding this problem and a more comprehensive account of the literature.

In general, to study $\mathrm{CH}^2(X)$ one spreads out $X$ to a smooth projective model $\mathcal{X} / \Oo$, where $\Oo$ is a localization of the ring of integers of $K$ at finitely many places, and considers the localization sequence 
$$\bigoplus_{\p \subset \Oo} \mathrm{Pic}(\mathcal{X}_{\Ff_{\p}}) \rightarrow \mathrm{CH}^2(\mathcal{X}) \rightarrow \mathrm{CH}^2(X) \rightarrow 0$$
(note that this can be continued to the left using $K$-theory). Assume from now on that $H^1(X, \Oo_X) = 0$, so, after further localizing $\Oo$ if necessary, we have $\mathrm{Pic}(\mathcal{X}_{\Ff_{\p}}) = \NS(\mathcal{X}_{\Ff_{\p}})$ for every $\p \subset \Oo$. The difficulty when $H^2(X, \Oo_X) \neq 0$ arises from the fact that the rank of $\NS(\mathcal{X}_{\Ff_{\p}})$ can jump for infinitely many places (as for example happens for K3 surfaces), and one is not able to control the image of $\bigoplus_{\p \subset \Oo} \mathrm{NS}(\mathcal{X}_{\Ff_{\p}})$ in $\mathrm{CH}^2(\mathcal{X})$. 

Let us call this image $\Sigma(\mathcal{X}) \cong \ker(\mathrm{CH}^2(\mathcal{X}) \rightarrow \mathrm{CH}^2(X) )$. This group appeared in \cite{MR1177312}, where the case $X = E \times E$ is studied, with $E/ \Q$ an elliptic curve. In particular, they show that $\Sigma(\mathcal{X})$ is finite when $K = \Q$ and $E$ has complex multiplication (there are only finitely many such curves over $\Q$). Moreover, the case where $E$ is an arbitrary elliptic curve over $\Q$ is also investigated. The proof starts by interpreting the new divisors modulo $\p$ as graphs of the Frobenius, and then uses the theory of modular curves and their reduction modulo $\p$ developed in \cite{MR337993} to control these cycles in $\mathrm{CH}^2(\mathcal{X})$. This is in fact one of the very few instances for which $H^2(X, \Oo_X) \neq 0$ and we nevertheless understand the role of vertical cycles in $\mathrm{CH}^2(\mathcal{X})$. See also \cite{MR1417619}, where these ideas are pushed further to prove more finiteness results regarding self products of elliptic curves over $\Q$. Our aim in this brief article is to generalize Mildenhall's finiteness to elliptic fibrations over $\Pp^1$:
\begin{thm} \label{Main thm}
Let $X/K$ be a smooth projective surface endowed with an elliptic fibration $\pi \colon X \rightarrow \Pp^1_K$ with at least one singular fibre and a section. Let $\mathcal{X}/ \Spec(\Oo)$ be a smooth projective model of $X$ over some open subset of $\Spec(\Oo_{K})$. Then $\Sigma(\mathcal{X})$ is a finitely generated abelian group. 
\end{thm}
This applies to many surfaces of interest, e.g., (some) K3 surfaces. Also note that $H^2(X, \Oo_X)$ can be arbitrarily large and that there are no restrictions on the number field $K$.
\begin{cor}
Let $X/K$ be as in the theorem above. Then $\mathrm{CH}^2(X)[n]$ is finite for every $n>0$ and the group $\mathrm{CH}^2(X)$ is finitely generated if and only if $\mathrm{CH}^2(\mathcal{X})$ is finitely generated.
\end{cor}
This follows from \cite{CT-R}[Theorem 1.1.], which says that $\mathrm{CH}^2(\mathcal{X})[n]$ is finite. To prove Theorem \ref{Main thm} we use the relation between the Néron--Severi groups and the Mordell--Weil groups of elliptic surfaces. Our idea is to realize any new section of $\pi_{\Ff_{\p}} \colon \mathcal{X}_{\Ff_{\p}} \rightarrow \Pp^1_{\Ff_{\p}}$ as (a component of) the degeneration of a $2$-section $D \subset X$ defined over $K$. This is made precise in Section \ref{final}. Once we do this, the result will follow almost immediately from the formal properties of Chow groups, specialization maps and the aforementioned finiteness results.
\section{Preliminaries}
\subsection{Elliptic fibrations}
We use the book \cite{MWL} as a reference on elliptic fibrations. We work with $\mathbb{P}^1_K$ as the base curve, where for the moment $K$ is an arbitrary field. Recall that a Weierstrass fibration $\pi' \colon X' \rightarrow \mathbb{P}^1_K$ is, by definition, a normal projective surface $X'$ together with a flat morphism $\pi'$ whose geometric fibres are either smooth curves of genus one or reduced plane cubic curves with a single node or a single cusp (i.e. nodal or cuspidal cubic curves). One also assumes that $\pi'$ admits a section whose image we regard as the zero section of the elliptic fibration.

To any Weierstrass fibration we can, by resolution of singularities, associate a minimal elliptic surface $\pi \colon X \rightarrow \mathbb{P}^1_K$. Thus $X$ is regular and the fibres of $\pi$ do not contain any $(-1)$-curves. The singular fibres of $\pi$ can be computed via the Tate algorithm, which we briefly explain. If $\Char(K) \neq 2,3$ we can write the Weierstrass fibration as
\[
y^2 = x^3 + a_4 x + a_6
\]
where $a_4, a_6 \in K(t)$. Now let $0, \infty \in \Pp^1(K)$ and let $U = \Pp^1 \setminus \{0\}$ and $U' = \Pp^1 \setminus \{\infty\}$, with $U = \Spec(K[t])$ and $U' = \Spec(K[t^{-1}])$. We can assume that $a_4, a_6 \in K[t]$ up to an admissible change of coordinates. Since $U$ has trivial class group, we can then find a global minimal Weierstrass equation over $U$. This means that either $v(a_4) < 4$ or $v(a_6) < 6$ for every valuation $v$ of $K(t)$ supported on $U$. 

Now let $d$ be the smallest integer such that $\deg(a_i) \leq di$ for $i = 4,6$. Given such a minimal Weierstrass equation over $U$, one can obtain the minimal Weierstrass form over $U'$ by seeing the $a_i$'s as sections of $\mathcal{O}(id)$. From now on we let
\[
P(x) = x^3 + a_4 x + a_6
\]
describe a minimal Weierstrass fibration over $U$ with $a_4, a_6 \in K[t]$. So if $t = t_0/t_1$ and we put
\[
a_i' = (t_1/t_0)^{id} \, a_i(t_0/t_1),
\]
we obtain two polynomials in $t^{-1}$, and the minimal Weierstrass equation over $U'$ is given by
\[
P'(x) = x^3 + a_4' x + a_6'.
\]
The Tate algorithm tells us that we can determine the singular fibres of $\pi$ over a closed point of $U$ (resp. $U'$) simply by looking at the values of $v(a_i)$, $v(\Delta)$ and $v(j)$ (resp. $v(a_i')$, $v(\Delta')$ and $v(j')$), where $\Delta = -16\bigl(4a_4^3 + 27 a_6^2\bigr)$ is the discriminant function, $j = -1728 \,\frac{(4a_4)^3}{\Delta}$ is the $j$-invariant and $v$ is the valuation corresponding to that closed point (see Tables 5.1 and 5.2 of \cite{MWL}).
\begin{defi}
Let $K$ be a number field and $\p \subset \Oo_{K}$ a prime. Let $P(x)$ describe a minimal Weierstrass equation. We say that $\p$ satisfies (*) if:
\begin{enumerate}
    \item the residue characteristic of $\p$ is not $2$ or $3$;
    \item all the rational functions $a_4, a_6, \Delta, j \in K(t)$ are integral at $\p$, and their reduction modulo $\p$ maintains the same factorization type (and the same holds for the analogous functions on the other chart). 
\end{enumerate}
\end{defi}
In other words, point (2) means the following. Let $v_{\p} \colon K(t) \rightarrow \Z$ be the valuation corresponding to the prime $\p \subset \Oo_K$ and let $R(t)$ be the ring of integers of $v_{\p}$. For $f \in R(t)$ we decompose
\[
f(t) = a \prod_i f_i^{\alpha_i}
\]
with $a\in R$ and $f_i \in R[t]$ irreducible and pairwise coprime and $v_{\p}(f_i) = 0$ for every $i$. We then say that $f$ maintains the same factorization type modulo $\p$ if the reduction $\bar{f}_i \in \Ff_{\p}[t]$ is separable, of the same degree as $f_i$, and if moreover $\bar{f}_i$ and $\bar{f}_j$ are coprime whenever $i \neq j$. 

Then all but finitely many primes satisfy (*), and the Tate algorithm implies that we can resolve the singularities of the Weierstrass model $\mathcal{X}'/R$ simultaneously over $R = \Oo_{\p}$. In particular, $X/K$ has good reduction at $\p$, and all its reductions in positive characteristic maintain the same configuration of singular fibres. Thus we have in this case a smooth projective model $\mathcal{X} / \Spec(\Oo)$, where $\Oo$ is the localization of $\Oo_{K}$ at the finitely many primes which do not satisfy (*), and an elliptic fibration $\tilde{\pi} \colon \mathcal{X} \rightarrow \mathbb{P}^1_{\Oo}$. For any $\p \subset \Oo$ we denote by $\pi_{\Ff_{\p}} \colon X_{\Ff_{\p}} \rightarrow \Pp^1_{\Ff_{\p}}$ the induced elliptic fibration over the residue field of $\p$.

\subsection{Specialization map on N\'{e}ron-Severi lattices}
For a prime $\p$ satisfying (*) we consider the specialization map
\[
\sigma_{\p} \colon \NS(X) \longrightarrow \NS(X_{\Ff_{\p}}).
\]
Since $X$ is fibred over $\Pp^1_K$ we have $\Pic^0(X) = 0$ and $H^1(X, \Oo_X) = 0$ as well (and the same is true over any field, see \cite{MWL}[Section 5.3]). Moreover, since $X$ has a rational point (because $\pi$ has a section), we also have, by standard arguments, that
\[
\Pic(X)= \Pic(\overline{X})^{G_K},
\]
where $G_K$ is the Galois group of $K$. Hence $\NS(X) = \NS(\overline{X})^{G_K}$. We also recall that the Néron--Severi group of an elliptic surface is always torsion-free (see \cite{MWL}[Theorem 6.4]). 

\begin{prop}
Assume that $e \leq p-1$, where $p$ is the residue characteristic of $\p$ and $e$ is the ramification index of $\p$ over $\Q$. Then $\sigma_{\p}$ is injective and has torsion-free cokernel. 
\end{prop}

\begin{proof}
By \cite{MR563468}[Theorem 4.2.1.2] (see also \cite{MR2948470}[Theorem 1.4]) the specialization map
\[
\Pic(X_{\bar{K}}) \longrightarrow \Pic(X_{\overline{\Ff_{\p}}})
\]
has torsion-free cokernel under the assumption. The result then follows because $\NS(X) = \NS(X_{\overline{K}})^{G_K}$ and similarly for $\NS(X_{\Ff_{\p}})$, and these are primitive subgroups of the corresponding Picard groups.  
\end{proof}

We say that a prime satisfies (**) if it satisfies (*) and, moreover, $e \leq p-1$. Thus again all but finitely many primes satisfy (**), and, up to further localizing $\Oo$, we can assume that $\sigma_{\p}$ has torsion-free cokernel for every $\p \subset \Oo$. Let now $\mathrm{Triv}(X) \subset \NS(X)$ be the primitive sublattice generated by the zero section and the components of the singular fibres. Then
\[
\NS(X_{\bar{K}})/ \mathrm{Triv}(X_{\bar{K}}) \cong \mathrm{MW}((X_{\bar{K}})_{\eta}),
\]
the Mordell--Weil group of $X$, viewed as an elliptic curve over the generic point $\eta$ of $\Pp^1_{\bar{K}}$. Under our assumption on $\Oo$ we also have
\[
\sigma_{\p} \colon \mathrm{Triv}(X_{\bar{K}}) \xrightarrow{\sim} \mathrm{Triv}({X_{\bar{\Ff_{\p}}}})
\]
for every prime $\p$ of $\mathcal{O}$. 
\begin{rmk}
Note that it may happen that some components of the singular fibres of $\pi$ are not defined over $K$ while they are defined over $\Ff_{\p}$. On the other hand, for a general $\pi$ the singular fibres are irreducible and this does not occur. 
\end{rmk}
\section{Double covers attached to new sections} \label{final}
Let now $R$ be any localization of $\Oo$ at to some prime $\p$ satisfying (**). We denote by $\mathcal{X}/R$ the model of $X$ over $R$ as in the previous section. We have a natural map $\mathcal{X} \rightarrow \mathcal{X}'$ (which is a sequence of blow-downs over $R$), where $\mathcal{X}'$ is the Weierstrass model associated to a minimal equation as before. Let $\p \subset R$ be the maximal ideal of $R$ and assume that $\pi_{\Ff_{\p}} \colon X_{\Ff_{\p}} \rightarrow \mathbb{P}^1_{\Ff_{\p}}$ acquires a new section $s_{\Ff_{\p}} \colon \Pp^1_{\Ff_{\p}} \rightarrow X_{\Ff_{\p}}$. Let $D_{\Ff_{\p}} \subset X_{\Ff_{\p}}$ be its associated Cartier divisor. Let now $\mathcal{U} = \Spec(\Oo[t]) \subset \Pp^1_{\Oo}$ and consider the open subset
\[
\Spec\bigl(\Ff_{\p}[t,x,y]/(y^2 - x^3 -\bar{a}_4 x - \bar{a}_6)\bigr)
\]
of $X_{\Ff_{\p}}'$ which maps to $U_{\Ff_{\p}} = \Spec(\Ff_{\p}[t])$ via $\pi_{\Ff_{\p}}$ (on the algebra side the bar denotes reduction modulo $\p$). The point at infinity then corresponds to the zero section, hence $D_{\Ff_{\p}}$ must intersect $\Spec(\Ff_{\p}[t,x,y]/(y^2 - x^3 -\bar{a}_4 x - \bar{a}_6))$ nontrivially. Thus we can represent the section $s_{\Ff_{\p}}$ over $U_{\Ff_{\p}}$ by some rational function $\bar{f} \in \Ff_{\p}(t)$ such that
\[
\bar{P}(\bar{f}) = \bar{f}^3 + \bar{a}_4 \bar{f} + \bar{a}_6 \in \Ff_{\p}(t)^2.
\]
Write $\bar{g}^2 = \bar{P}(\bar{a})$ for some $\bar{g} \in \Ff_{\p}(t)$, so that $s_{\Ff_{\p}}$ is given in the $(x,y)$-coordinates by $(\bar{f}, \bar{g})$ and its negation $-s_p$ by $(\bar{f}, -\bar{g})$.

Consider now the function field $K(t)$ and the valuation $v_{\p} \colon K(t) \rightarrow \Z$. Then we can write
\[
R(t) \coloneqq \bigg\{ \frac{f}{g} \;\bigg|\; f,g \in R[t],\ v_{\p}(g) = 0 \bigg\}.
\]
Pick any $f \in R(t)$ which lifts $\bar{f}$ and consider $G \coloneqq P(f)$. Note that $G \in R(t)$ as well, and that $G$ is a lift of $\bar{g}^2$. On the other hand, $G$ cannot itself be a square, for otherwise $(f,g)$ would describe a section of $X$ lifting $D_{\Ff_{\p}}$, where $g^2 = G$, which is impossible under our assumption. 

So it makes sense to consider the field extension $L = K(t)[g]/(g^2 - G)$ and define $\mathcal{C}(D_{\Ff_{\p}})$ as the integral closure of $\Pp^1_R$ in $L$. Note that this depends on the choice of the lift $f$, and not only on $D_{\Ff_{\p}}$, but this is not relevant for us. Since integral closure commutes with localization, the generic fibre of $\mathcal{C}(D_{\Ff_{\p}})$ is the integral closure of $\Pp^1_K$ in $L$. Also, note that $L$ may split over $\bar{K}$, e.g.\ if $g$ becomes a square over an extension of $K$. This can happen only if $\NS(X_{\bar{K}})^{G_K} \neq \NS(X_{\bar{K}})$ and $D_{\Ff_{\p}}$ does not lift over $K$ but lifts over an extension of $K$. Moreover, it is also necessary that the lift $f$ is chosen such that $P(f)$ is a geometric square. So in order to have $\Spec(L)$ geometrically irreducible in general we need only to show that $P(f)$ can always be chosen so that it is not a square in $\bar{K}(t).$
\begin{lemma}
There are infinitely many lifts $f$ of $\bar{f}$ such that $P(f)$ is not a square in $\overline{K}(t)$.
\end{lemma}

\begin{proof}
Pick any lift $f = f_0/f_1$ of $\bar{f} = \bar{f}_{0}/ \bar{f}_{1}$, where
$\bar{f}_{0}, \bar{f}_{1} \in \Ff_{\p}[t]$ are coprime and
$f_0,f_1 \in R[t]$ have the same degrees as $\bar{f}_{0}$ and $\bar{f}_{1}$
respectively. Then $f_0$ and $f_1$ are also coprime and
$v_{\p}(f_0) = v_{\p}(f_1) = 0$. For a local parameter $\pi$ of $\p$ we
consider lifts of the form
\[
f_\lambda = \frac{f_0 + \pi \lambda}{f_1}, \qquad \lambda \in R.
\]
Then
\[
f_1^3\,P\bigl((f_0 + \pi \lambda)/f_1\bigr)
= (f_0 + \pi \lambda)^3 + a_4 f_1^2 (f_0+\pi \lambda)+ a_6 f_1^3
\eqqcolon J_{\lambda},
\]
and the right-hand side is coprime to $f_1$ since $f_0$ and $f_1$ are coprime
modulo $\p$. Thus it is enough to show that there are infinitely many
choices of $\lambda$ such that $J_{\lambda}$ is squarefree in $\overline{K}[t]$.
If now, for a ring $S$, we denote by $S[t]_{N} \cong S^{N+1}$ the
$S$-module of polynomials of degree $\leq N$, then the subset of
squarefree polynomials in $S[t]_{N}$ is a Zariski open subset defined
over $\Z$. Hence it is enough to show that $J_{\lambda}$ is squarefree over $K[t]$
for infinitely many $\lambda$. Now note that
\[
Q(x,t) = x^3 + a_4 f_1^2 x+ a_6 f_1^3
\]
satisfies $Q(f_1 x , t) = f_1^3 P(x,t)$ and that $P(x,t)$ is irreducible
in $K[x,t]$. If, for a contradiction, we had $Q(x,t) = Q_0(x,t) Q_1(x,t)$,
then
\[
Q_0(f_1x,t)\, Q_1(f_1x,t) = f_1^3 P(x,t),
\]
which would imply, by degree considerations in $x$, that either $Q_0$ or
$Q_1$ is constant in $x$, which is impossible unless the other factor
is a unit. Thus $Q(x,t)$ is irreducible, and hence also
$Q(f_0 + \pi x,t)$ must be irreducible. Finally, by Hilbert
irreducibility, we have infinitely many $x_0 \in K$ such that
$Q(f_0 + \pi x_0,t)$ remains irreducible over $K$, and in particular, squarefree.
\end{proof}
So we can choose $L$ such that $\Spec(L)$ is geoemtrically irreducible and the generic fiber of $\mathcal{C}(D_{\Ff_{\p}})$ is a smooth projective hyperelliptic curve.
\begin{prop}
The induced map $\phi \colon \mathcal{C}(D_{\Ff_{\p}}) \rightarrow \Pp^1_{R}$ is finite (of degree $2$) and flat. The special fibre of $\mathcal{C}(D_{\Ff_{\p}})$ splits into two components $(\Pp^1_{\Ff_{\p}})^{+} \cup (\Pp^1_{\Ff_{\p}})^{-}$ meeting at the zeroes and poles of $\bar{g}$. The two components are swapped by the reduction modulo $\p$ of the hyperelliptic involution $\iota \colon \mathcal{C}(D_{\Ff_{\p}}) \xrightarrow{\sim} \mathcal{C}(D_{\Ff_{\p}})$.
\end{prop}

\begin{proof}
The map is finite of degree $2$ by construction. It is also flat by miracle flatness, since $\mathcal{C}(D_{\Ff_{\p}})$ is normal and of dimension two, hence automatically Cohen–Macaulay (see the beginning of \cite{MR1705977}[Section 1.2], where the same construction is considered). The claim about the special fibre follows from the definition of $L$, since $G$ becomes a square modulo $\p$. Note also that taking the quotient by $\iota$ commutes with taking the special fibre (see \cite{MR1705977}[Remark 1.7]) since the residue characteristic is odd.
\end{proof}Next, define $\mathcal{X}(D_{\Ff_{\p}})$ as the base change
\[
\begin{tikzcd}
	{\mathcal{X}(D_{\Ff_{\p}})} & {\mathcal{X}} \\
	{\mathcal{C}(D_{\Ff_{\p}})} & {\Pp^1_R}
	\arrow["\psi", from=1-1, to=1-2]
	\arrow["\sigma", from=1-1, to=2-1]
	\arrow["\pi"', from=1-2, to=2-2]
	\arrow["\phi", from=2-1, to=2-2]
\end{tikzcd}
\]
So we have a natural finite flat map $\psi \colon \mathcal{X}(D_{\Ff_{\p}}) \rightarrow \mathcal{X}$ induced by $\psi$ and the left vertical map induces an elliptic fibration $\sigma \colon \mathcal{X}(D_{\Ff_{\p}}) \rightarrow \mathcal{C}(D_{\Ff_{\p}})$. Note that this is not necessarily a minimal elliptic fibration (in particular its generic fibre may be singular), but it is birational to one. In any case, by construction, we have a new section of $\sigma$ described in the $(x,y)$-coordinates by $(f,g)$. 

Note that both $f$ and $g$ are $\p$-integral in $L$ (since $f \in R(t)$ by construction and $g$ is a solution of the degree-$2$ equation $y^2 = G$ with $G \in R(t)$). Thus, its image yields a prime divisor $\mathcal{D}^{+} \subset \mathcal{X}(D_{\Ff_{\p}})$, which is the Zariski closure of the section $D^{+} \subset X(D_{\Ff_{\p}})$ induced by $(f,g)$ on the generic fibre. The negation $(f,-g)$ induces another prime divisor $\mathcal{D}^{-} = \iota(\mathcal{D}^{+})$. Finally, up to renaming, we can also assume that the special fibre of $\mathcal{D}^{+}$ coincides with the section $D_{\Ff_{\p}} \subset X_{\Ff_{\p}}$, seen as a divisor on the component $(X_{\Ff_{\p}})^{+} \cong X_{\Ff_{\p}}$ of $\mathcal{X}(D_{\Ff_{\p}})_{\Ff_{\p}}$, and similarly for $\mathcal{D}^{-}$.

\subsection{Chow groups and Proof of Theorem \ref{Main thm}}
Chow groups of schemes over Dedekind domains and their functoriality are studied in Fulton's book \cite{MR1644323}[Chapter 20]. We recall the main results that we are going to use (keeping the notation from the previous section). Let $\mathcal{X}$ be a separated scheme of finite type over $\mathcal{O}$ (or $R$) with generic fibre $X$ and special fibres $X_{\Ff_{\p}}$ (both may be empty in Fulton's setting). If $Z \subset \mathcal{X}$ is a closed integral subscheme we let
\[
\dim_{\Oo}(Z) =
\begin{cases}
    \dim(Z_{K}) & \text{if } Z \text{ is flat over } \Oo, \\[4pt]
    \dim(Z) - 1 & \text{if } Z \text{ is vertical}.
\end{cases}
\]
The group $\mathrm{CH}_i(\mathcal{X} / \Oo)$ is then defined as the free abelian group generated by integral closed subschemes of relative dimension $i$ modulo rational equivalence (thus $\mathrm{CH}^2(\mathcal{X}) = \mathrm{CH}_0(\mathcal{X}/ \Oo)$ in this notation). The specialization map
\[
\sigma_p \colon \mathrm{CH}_i(X) \rightarrow \mathrm{CH}_i(X_{\Ff_{\p}})
\]
is defined in \cite{MR1644323}[Section 20.3] in full generality. Moreover, if $Z \subset X$ is a prime cycle and $\mathcal{Z} \subset \mathcal{X}$ denotes its Zariski closure, which is flat over $\Oo$, then
\[
\sigma_p(Z) = [\mathcal{Z}_{\Ff_{\p}}] \in \mathrm{CH}_i(X_{\Ff_{\p}}),
\]
i.e. $\sigma_p(Z)$ is the element of $\mathrm{CH}_i(X_{\Ff_{\p}})$ associated to the special fibre of $\mathcal{Z}$, seen as a closed subscheme of $X_{\Ff_{\p}}$. The following lemma is well-known:

\begin{lemma} \label{lemma zero chow}
Let $j \colon \mathcal{Z} \hookrightarrow \mathcal{X}$ be the inclusion of an integral closed subscheme. Assume that $\mathcal{Z}$ is flat over $\Oo$ and of relative dimension $i$. Then the cycle associated to the special fibre $\mathcal{Z}_{\Ff_{\p}}$ is zero in $\mathrm{CH}_{i-1}(\mathcal{X}/ \Oo)$.
\end{lemma}

\begin{proof}
Let $\pi \in \Oo$ be a local parameter for $\p$ and, if necessary, localize $\Oo$ further to obtain $\Oo'$ so that $\pi$ has only $\p$ as a simple zero on $\Spec(\Oo')$. Let now $\pi_{| \mathcal{Z}}$ be considered as a rational function on $\mathcal{Z}$. Then by the definition of rational equivalence
\[
\mathrm{div}(\pi_{| \mathcal{Z}}) =  [\mathcal{Z}_{\Ff_{\p}}] = 0
\]
in the Chow group $\mathrm{CH}_{i-1}(\mathcal{Z}/ \Oo')$. Hence, by the fact that proper pushforward respects rational equivalence, we get that $[\mathcal{Z}_{\Ff_{\p}}] = 0$ also in $\mathrm{CH}_{i-1}(\mathcal{X}/ \Oo')$. Finally, this is also zero in $\mathrm{CH}_{i-1}(\mathcal{X}/ \Oo)$ by functoriality.
\end{proof}
Let now $X$ be an elliptic surface with a model $\mathcal{X} / R$ as in the previous section. Let $\psi \colon \mathcal{X}(D_{\Ff_{\p}}) \rightarrow \mathcal{X}$ be the $2\!:\!1$ covering corresponding to some new section $D_{\Ff_{\p}} \in \NS(X_{\Ff_{\p}})$. Let
\[
\psi_{\Ff_{\p}} \colon \mathcal{X}(D_{\Ff_{\p}})_{\Ff_{\p}} = (X_{\Ff_{\p}})^{+} \cup (X_{\Ff_{\p}})^{-} \longrightarrow X_{\Ff_{\p}}
\]
be the induced (finite, flat) morphism between special fibres. By construction
\[
\psi_{\Ff_{\p}}^{*}(D_{\Ff_{\p}}) = \sigma_{\p}(D^{+} + D^{-}),
\]
where $D^{+} \subset X(D_{\Ff_{\p}}) = \mathcal{X}(D_{\Ff_{\p}})_{K}$ is the generic fibre of $\mathcal{D}^{+} \subset \mathcal{X}(D_{\Ff_{\p}})$. Thus the previous lemma implies that $[D_{\Ff_{\p}}] \in \mathrm{CH}_0(\mathcal{X} / R)$ satisfies
\[
0 = \psi^{*}([D_{\Ff_{\p}}]) \in \mathrm{CH}_0(\mathcal{X}(D_{\Ff_{\p}}) / R).
\]
Hence
\[
0 = \psi_{*}(\psi^{*}[D_{\Ff_{\p}}]) = 2 \cdot [D_{\Ff_{\p}}] \in \mathrm{CH}_{0}(\mathcal{X}/R).
\]
This shows that the element $(\dots, 0, D_{\Ff_{\p}}, 0, \dots) \in \bigoplus_{\p \subset \Oo} \NS(X_{\Ff_{\p}})$ has image in $\mathrm{CH}_0(\mathcal{X}/ \Oo)[2]$. Finally, all the components of the singular fibres of $\pi \colon X \rightarrow \Pp^1_K$ are defined over some field extension $K'/K$ of degree $N$. Hence, if $(\dots, 0, F, 0, \dots) \in \bigoplus_{\p \subset \Oo} \NS(X_{\Ff_{\p}})$ describes a component of some singular fibre which is not defined over $K$, we must also have that
\[
N \cdot (\dots, 0, F, 0, \dots) = 0 \quad \text{in } \mathrm{CH}_0(\mathcal{X}/ \Oo)
\]
by a standard push–pull argument and so the image of $\bigoplus_{\p \subset \Oo} \NS(X_{\Ff_{\p}})$ is contained in $\mathrm{CH}_0(\mathcal{X}/ \Oo)[2N] = \mathrm{CH}^2(\mathcal{X})[2N]$, which is a finite group \cite{CT-R}[Theorem 1.1.]. The result then follows since there are only finitely many primes which do not satisfy (**) and since $\NS(X_{\Ff_{\p}})$ is always finitely generated.

\bibliographystyle{alpha}
\bibliography{ref}
\end{document}